\documentclass[12pt, reqno]{amsart}
\usepackage{amsmath, amsthm, amscd, amsfonts, amssymb, graphicx, color}
\usepackage[bookmarksnumbered, colorlinks, plainpages]{hyperref}

\hypersetup{colorlinks=true,linkcolor=red, anchorcolor=green,
citecolor=cyan, urlcolor=red, filecolor=magenta, pdftoolbar=true}
\input{mathrsfs.sty}
\newtheorem{theorem}{Theorem}[section]
\newtheorem{lemma}[theorem]{Lemma}
\newtheorem{proposition}[theorem]{Proposition}
\newtheorem{corollary}[theorem]{Corollary}
\theoremstyle{definition}

\theoremstyle{remark}
\newtheorem{remark}[theorem]{Remark}
\numberwithin{equation}{section}

\begin{document}

\title[Orthogonally $a$-Jensen mappings on $C^*$-modules]
{Orthogonally $a$-Jensen mappings on $C^*$-modules}

\author[A. Zamani]{Ali Zamani}

\address{Department of Mathematics, Farhangian University, Tehran, Iran}
\email{zamani.ali85@yahoo.com}



\begin{abstract}
We investigate the representation of the so-called orthogonally
$a$-Jensen mappings acting on $C^*$-modules.
More precisely, let $\mathfrak{A}$ be a unital $C^*$-algebra with the unit $1$,
let $a \in \mathfrak{A}$ be fixed such that $a, 1-a$ are invertible
and let $\mathscr{E}, \mathscr{F}, \mathscr{G}$ be inner product $\mathfrak{A}$-modules.
We prove that if there exist additive mappings $\varphi, \psi$
from $\mathscr{F}$ into $\mathscr{E}$ such that
$\big\langle \varphi(y), \psi(z)\big\rangle=0$ and
$a \big\langle \varphi(y), \varphi(z)\big\rangle a^\ast
= (1 - a)\big\langle \psi(y), \psi(z)\big\rangle (1 - a)^\ast$ for all $y, z\in \mathscr{F}$,
then a mapping $f: \mathscr{E} \to \mathscr{G}$ is orthogonally $a$-Jensen
if and only if it is of the form $f(x) = A(x) + B(x, x) +f(0)$
for $x\in \mathscr{K} := \varphi(\mathscr{F})+\psi(\mathscr{F})$,
where $A: \mathscr{E} \to \mathscr{G}$ is an $a$-additive mapping on $\mathscr{K}$
and $B$ is a symmetric $a$-biadditive orthogonality preserving mapping
on $\mathscr{K}\times \mathscr{K}$.
Some other related results are also presented.
\end{abstract} \maketitle

\section{{\rm \bf Introduction}}
Orthogonally functionals on an inner product space when the orthogonality
is the ordinary one have been considered by Pinsker \cite{pi}.
Next Sundaresan \cite{Su} generalized the result of Pinsker to arbitrary
Banach spaces equipped with the Birkhoff--James orthogonality.
In the recent decades, mappings satisfying a functional equation under
some orthogonality conditions have been investigated by several mathematicians,
who have presented many interesting results and applications,
see, e. g., \cite{C.L.Z, C.L.W, F.M.P, I.T.Y, M.Z, P.P.V, Ratz}.

Jensen \cite{J} first studied the functions satisfying the condition
$f(\frac{x + y}{2}) =\frac{f(x) + f(y)}{2}$. It is easy to see that every continuous
Jensen function on $\mathbb{C}$ is affine in the sense that $f-f(0)$ is additive.
The Jensen functional equation has been extensively studied from many aspects
by many mathematicians, see, e.g., \cite{Ng, S.K} and the references therein.

Let us recall some definitions and introduce our notation. An inner product module
over a $C^{*}$-algebra $\mathfrak{A}$ is a (right) $\mathfrak{A}$-module $\mathscr{E}$
equipped with an $\mathfrak{A}$-valued inner product $\langle\cdot,\cdot\rangle$, which
is $\mathbb{C}$-linear and $\mathfrak{A}$-linear in the second variable and has the properties
$\langle x, y\rangle^*=\langle y, x\rangle$ as well as $\langle x, x\rangle \geq 0$ with equality
if and only if $x = 0$. An inner product $\mathfrak{A}$-module $\mathscr{E}$ is called a Hilbert
$\mathfrak{A}$-module if it is complete with respect to the norm $\|x\|=\|\langle x, x\rangle\|^{\frac{1}{2}}$.

Although inner product $C^{*}$-modules generalize inner product spaces by allowing
inner products to take values in an arbitrary $C^{*}$-algebra instead of the
$C^{*}$-algebra of complex numbers, but some fundamental properties of inner product
spaces are no longer valid in inner product $C^{*}$-modules. For example,
not each closed submodule of an inner product $C^{*}$-module is complemented.
Therefore, when we are studying in inner product $C^{*}$-modules,
it is always of some interest to find conditions to obtain the results analogous
to those for inner product spaces. We refer the reader to \cite{Man} for more
information on the theory $C^*$-algebras and the structure of Hilbert $C^{*}$-modules.

Let $\mathscr{E}$ and $\mathscr{F}$ be two inner product $\mathfrak{A}$-modules.
A morphism between inner product $\mathfrak{A}$-modules $\mathscr{E}$ and
$\mathscr{F}$ is a mapping $\varphi:\mathscr{E}\longrightarrow \mathscr{F}$
satisfying $\langle \varphi(x), \varphi(y)\rangle=\langle x, y\rangle$ for all $x, y\in \mathscr{E}$.
A mapping $t:\mathscr{E}\longrightarrow \mathscr{F}$ is called adjointable
if there exists a mapping $s:\mathscr{F}\longrightarrow \mathscr{E}$ such that
$\langle tx, y\rangle=\langle x, sy\rangle$ for all $x\in \mathscr{E}, y\in \mathscr{F}$.
The unique mapping $s$ is denoted by $t^*$ and is called the adjoint of $t$.
Furthermore, inner product $\mathfrak{A}$-modules $\mathscr{E}$ and $\mathscr{F}$
are unitarily equivalent ( and we write $\mathscr{E}\thicksim \mathscr{F}$)
if there exists an adjointable mapping $u:\mathscr{E}\longrightarrow \mathscr{F}$
such that $u^* u=id_\mathscr{E}$ and $uu^*=id_\mathscr{F}$.
A closed submodule $\mathscr{G}$ of an inner product $\mathfrak{A}$-module $\mathscr{E}$
is said to be orthogonally complemented if $\mathscr{G}\oplus \mathscr{G}^\perp=\mathscr{E}$,
where $\mathscr{G}^\perp=\{x\in \mathscr{E}:\,\langle x, y\rangle=0 \,\,\mbox{for all}\,\,y\in \mathscr{G}\}$.
A closed submodule $\mathscr{K}$ of an inner product $\mathfrak{A}$-module $\mathscr{E}$ is said to be
fully complemented if $\mathscr{K}$ is orthogonally complemented and $\mathscr{K}^\perp\thicksim \mathscr{E}$.
Note that the theory of inner product $C^{*}$-modules is quite different from that of inner product spaces.
For example, not any closed submodule of an inner product $C^{*}$-module is complemented
and there might exist bounded $\mathfrak{A}$-linear operators that are not adjointable.

Throughout the paper let $\mathfrak{A}$ be a unital $C^*$-algebra with the unit $1$ and let
$\mathscr{E}, \mathscr{F}, \mathscr{G}$ be inner product $\mathfrak{A}$-modules.
We fix an element $a \in \mathfrak{A}$ such that $a, 1 - a$ are invertible.
For instance, $a$ can be an element of $\mathfrak{A}$ satisfying $0<a<1$,
where the order $c<d$ in $\mathfrak{A}$ means that $c, d$ are self-adjoint
and the spectrum of $d-c$ is contained in $(m, \infty]$ for some positive number $m$.
An additive mapping $A :\mathscr{E}\longrightarrow \mathscr{G}$
is called $a$-additive if $A(ax)=a A(x)$ for all $x\in \mathscr{E}$.
A biadditive mapping $B: \mathscr{E} \times \mathscr{E}\to \mathscr{G}$
is called $a$-biadditive if $B(ax, ax) = a B(x, x)$ and $B\big((1-a)x, (1-a)x\big) = (1-a) B(x, x)$
for all $x\in \mathscr{E}$.
It is symmetric if $B(x, y) = B(y, x)$ for all $x, y\in \mathscr{E}$.
Furthermore, $B$ is said to be orthogonality preserving if for all $x, y\in \mathscr{E}$,
$$\langle x, y\rangle=0\,\Longrightarrow\,B(x, y)=0.$$
A mapping $Q:\mathscr{E}\longrightarrow \mathscr{G}$ is said to be quadratic
if it satisfies the so-called quadratic functional equation
$$Q(x+y) + Q(x-y) = 2Q(x) + 2Q(y) \qquad (x, y\in \mathscr{E}).$$
Clearly any biadditive mapping is quadratic. A mapping $f: \mathscr{E} \to \mathscr{G}$
is called orthogonally $a$-Jensen if
\begin{align}\label{id.200}
\langle x, y\rangle=0\,\Longrightarrow\,f\big(ax + (1 - a)y\big) = a f(x) + (1 - a) f(y) \qquad (x, y\in \mathscr{E}).
\end{align}
In particular if $p\in(0, 1)$, with $a = p1$ the mapping $f$
satisfying (\ref{id.200}) is said to be orthogonally $p$-Jensen.
Further if $p=\frac{1}{2}$ we say that $f$ is orthogonally Jensen.

In this paper, we investigate the representation of the so-called orthogonally
$a$-Jensen mappings acting on inner product $C^*$-modules.
More precisely, we prove that if there exist additive mappings
$\varphi, \psi$ from $\mathscr{F}$ into $\mathscr{E}$ such that $\big\langle \varphi(y), \psi(z)\big\rangle=0$
and $a \big\langle \varphi(y), \varphi(z)\big\rangle a^\ast = (1 - a)\big\langle \psi(y), \psi(z)\big\rangle (1 - a)^\ast$
for all $y, z\in \mathscr{F}$, then a mapping $f: \mathscr{E} \to \mathscr{G}$ is orthogonally $a$-Jensen
if and only if it is of the form $f(x) = A(x) + B(x, x) +f(0)$ for
$x\in \mathscr{K} := \varphi(\mathscr{F})+\psi(\mathscr{F})$,
where $A: \mathscr{E} \to \mathscr{G}$ is an $a$-additive mapping on $\mathscr{K}$
and $B$ is a symmetric $a$-biadditive orthogonality preserving mapping on $\mathscr{K}\times \mathscr{K}$.
In addition, we show that if $\mathscr{F}$ is a fully complemented submodule of $\mathscr{E}$
and $f$ is orthogonally Jensen, then $f$ is of the form $f(x) = A(x)+f(0)$ for $x\in \mathscr{F}$.
\section{{\rm \bf Main results}}
We start our work with the following lemmas.
The first lemma follows immediately from (\ref{id.200}).
\begin{lemma}\label{lm.201}
If $f:\mathscr{E}\longrightarrow \mathscr{F}$ is orthogonally $a$-Jensen, then
\begin{itemize}
\item[(i)] $a f( a^{-1}x) + (1 - a) f(0) = f(x)$
\item[(ii)] $a f(0) + (1 - a) f\big( (1 - a)^{-1}x\big) = f(x)$
\item[(iii)] $f( a^{-1}x) + a^{-1}(1 - a) f(0) = a^{-1} f(x)$
\item[(iv)] $(1 - a)^{-1}a f(0) + f\big( (1 - a)^{-1}x\big) = (1 - a)^{-1} f(x)$
\item[(v)] $(1 - a)^{-1}a f(x) + f(0) = (1 - a)^{-1} f(ax)$
\item[(vi)] $f(0) + a^{-1}(1 - a) f(x) = a^{-1} f\big((1 - a)x\big)$
\end{itemize}
for every $x \in \mathscr{E}$.
\end{lemma}
\begin{lemma}\label{lm.202}
Suppose that there exist additive mappings $\varphi, \psi:\mathscr{F}\longrightarrow \mathscr{E}$
such that $\big\langle \varphi(z), \psi(w)\big\rangle=0$ and
$a \big\langle \varphi(z), \varphi(w)\big\rangle a^\ast = (1 - a) \big\langle \psi(z), \psi(w)\big\rangle (1 - a)^\ast$
for all $z, w\in \mathscr{F}$. If $f:\mathscr{E}\longrightarrow \mathscr{G}$
is orthogonally $a$-Jensen, then
\begin{align*}
a f\big(\varphi(x) + \varphi(y)\big) &+ (1 - a) f\big(\psi(x) -\psi(y)\big)
\\&= a \Big[f\big(\varphi(x)\big) + a^{-1}(1 - a) f\big(\psi(x)\big) - (1 - a)a^{-1} f(0)\Big]
\\&\hspace{2cm} + (1 - a) \Big[(1 - a)^{-1}a f\big(\varphi(y)\big) - (1 - a)^{-1}a f(0) + f\big(\psi(-y)\big)\Big].
\end{align*}
for every $x, y\in \mathscr{F}$.
\end{lemma}
\begin{proof}
We have
\begin{align}\label{id.202.1}
\Big\langle \varphi(x) &+ a^{-1}(1 - a)\psi(x), (1 - a)^{-1}a \varphi(y) - \psi(y)\Big\rangle\nonumber
\\&= \big\langle \varphi(x), \varphi(y)\big\rangle \big((1 - a)^{-1}a\big)^\ast - \big\langle \varphi(x), \psi(y)\big\rangle \nonumber
\\&\hspace{2.5cm} + a^{-1}(1 - a) \big\langle \psi(x), \varphi(y)\big\rangle \big((1 - a)^{-1}a\big)^\ast
- a^{-1}(1 - a) \big\langle \psi(x) , \psi(y) \big\rangle \nonumber
\\&= \big\langle \varphi(x), \varphi(y)\big\rangle \big((1 - a)^{-1}a\big)^\ast
- a^{-1}(1 - a) \big\langle \psi(x) , \psi(y) \big\rangle \nonumber
\\&\hspace{3cm}\Big(\mbox{since $\big\langle \varphi(x), \psi(y)\big\rangle=0$ and
$\big\langle \psi(x), \varphi(y)\big\rangle = \big\langle \varphi(y), \psi(x)\big\rangle^* = 0$}\Big) \nonumber
\\&= \big\langle \varphi(x), \varphi(y)\big\rangle \big((1 - a)^{-1}a\big)^\ast
- a^{-1}(1 - a) \Big[(1 - a)^{-1}a \big\langle \varphi(x), \varphi(y)\big\rangle \big((1 - a)^{-1}a\big)^\ast\Big] \nonumber
\\&\hspace{4.5cm}\Big(\mbox{since
$(1 - a) \big\langle \psi(x), \psi(y)\big\rangle (1 - a)^\ast = a \big\langle \varphi(x), \varphi(y)\big\rangle a^\ast$}\Big) \nonumber
\\&=0
\end{align}
for every $x, y\in \mathscr{F}$. Therefore, we arrive at
\begin{align*}
a f\big(\varphi(x) &+ \varphi(y)\big) + (1 - a) f\big(\psi(x)-\psi(y)\big)
\\&= a f\big(\varphi(x+y)\big) + (1 - a) f\big(\psi(x-y)\big)
\\&= f\Big(a \varphi(x+y) + (1 - a) \psi(x-y)\Big)
\\&\hspace{2.7cm}\Big(\mbox{since $\big\langle \varphi(x+y), \psi(x-y)\big\rangle=0$ and $f$ is orthogonally $a$-Jensen}\Big)
\\&= f\Big(a\big[\varphi(x) + a^{-1}(1 - a)\psi(x)\big] + (1 - a)\big[(1 - a)^{-1}a \varphi(y) - \psi(y)\big]\Big)
\\&= a f\Big(\varphi(x) + a^{-1}(1 - a)\psi(x)\Big) + (1 - a) f\Big((1 - a)^{-1}a \varphi(y) - \psi(y)\Big)
\\&\hspace{5cm}\Big(\mbox{since $f$ is orthogonally $a$-Jensen and (\ref{id.202.1}) holds}\Big)
\\&= a f\Big(a\big[a^{-1}\varphi(x)\big] + (1 - a)\big[(1 - a)^{-1}a^{-1}(1 - a)\psi(x)\big]\Big)
\\&\hspace{2cm}+ (1 - a) f\Big(a \big[a^{-1}(1 - a)^{-1}a\varphi(y)\big] + (1 - a)\big[-(1 - a)^{-1} \psi(y)\big]\Big)
\\&= a\Big[a f\Big(a^{-1}\varphi(x)\Big) + (1 - a) f\Big((1 - a)^{-1}a^{-1}(1 - a)\psi(x)\Big)\Big]
\\&\hspace{2cm}+ (1 - a)\Big[a f\Big(a^{-1}(1 - a)^{-1}a\varphi(y)\Big) + (1 - a) f\Big(-(1 - a)^{-1} \psi(y)\Big)\Big]
\\&\hspace{1cm}\Big(\mbox{since $\big\langle a^{-1}\varphi(x), (1 - a)^{-1}a^{-1}(1 - a)\psi(x)\big\rangle = 0$},
\\&\hspace{5cm} \big\langle a^{-1}(1 - a)^{-1}a\varphi(y), -(1 - a)^{-1} \psi(y)\big\rangle = 0
\\&\hspace{8cm}\mbox{and $f$ is orthogonally $a$-Jensen}\Big)
\\&= a \Big[f\big(\varphi(x)\big) - (1 - a) f(0) + f\big(a^{-1}(1 - a) \psi(x)\big) - a f(0)\Big]
\\&\hspace{1cm}+ (1 - a) \Big[f\big((1 - a)^{-1}a \varphi(y)\big) - (1 - a) f(0) + f\big(-\psi(y)\big) - a f(0)\Big]
\\&\hspace{8cm}\Big(\mbox{by Lemma \ref{lm.201}\,(i) and (ii)}\Big)
\\&= a\Big[f\big(\varphi(x)\big) - (1 - a) f(0) + a^{-1} f\big((1 - a) \psi(x)\big)
\\&\hspace{5cm}- a^{-1}(1 - a) f(0) - a f(0)\Big]
\\&\hspace{3cm}+ (1 - a) \Big[(1 - a)^{-1} f\big(a \varphi(y)\big) - (1 - a)^{-1}a f(0)
\\&\hspace{6cm} - (1 - a) f(0) + f\big(-\psi(y)\big) - a f(0)\Big]
\\&\hspace{8cm}\Big(\mbox{by Lemma \ref{lm.201}\,(iii) and (iv)}\Big)
\\&= a \Big[f\big(\varphi(x)\big) - (1 - a) f(0) + f(0) + a^{-1}(1 - a) f\big(\psi(x)\big)
\\&\hspace{7cm} - (1 - a)a^{-1} f(0) - a f(0)\Big]
\\&\hspace{2cm} + (1 - a) \Big[(1 - a)^{-1}a f\big(\varphi(y)\big) + f(0)
\\&\hspace{4.5cm} - (1 - a)^{-1}a f(0) - (1 - a) f(0) + f\big(\psi(-y)\big) - a f(0)\Big]
\\&\hspace{8cm}\Big(\mbox{by Lemma \ref{lm.201}\,(v) and (vi)}\Big).
\end{align*}
From this it follows that
\begin{align*}
a f\big(\varphi(x) + \varphi(y)\big) &+ (1 - a) f\big(\psi(x) -\psi(y)\big)
\\&= a \Big[f\big(\varphi(x)\big) + a^{-1}(1 - a) f\big(\psi(x)\big) - (1 - a)a^{-1} f(0)\Big]
\\&\hspace{2cm} + (1 - a) \Big[(1 - a)^{-1}a f\big(\varphi(y)\big) - (1 - a)^{-1}a f(0) + f\big(\psi(-y)\big)\Big]
\end{align*}
and the lemma is proved.
\end{proof}
\begin{remark}
The condition that additive mappings $\varphi, \psi$ satisfying $\big\langle \varphi(x), \psi(y)\big\rangle=0$
and $a \big\langle \varphi(x), \varphi(y)\big\rangle a^\ast = (1 - a) \big\langle \psi(x), \psi(y)\big\rangle (1 - a)^\ast$
is not restrictive. In fact, there are non-trivial concrete examples of additive mappings satisfying this condition.
A non-trivial example can be given in $\ell^2$ by $a = 1-p$ with $p\in(0, 1)$ and
$$\begin{cases}
\varphi, \psi:\ell^2\longrightarrow \ell^2 &\\
\varphi(\{a_n\}) = \big(\frac{1}{1-p}a_1, 0, \frac{1}{1-p}a_2, 0, \frac{1}{1-p}a_3, 0,\cdots\big) &\\
\psi(\{a_n\}) = \big(0, \frac{1}{p}a_1, 0, \frac{1}{p}a_2, 0, \frac{1}{p}a_3, 0,\cdots\big).
\end{cases}$$
One can easily observe that $\big\langle \varphi(\{a_n\}), \psi(\{b_n\})\big\rangle=0$
and
$$a \big\langle \varphi(\{a_n\}), \varphi(\{b_n\})\big\rangle a^\ast
= (1 - a) \big\langle \psi(\{a_n\}), \psi(\{b_n\})\big\rangle (1 - a)^\ast = \sum_{n = 1}^\infty a_n b_n.$$
\end{remark}
The following auxiliary results are needed in our investigation.
\begin{proposition}\label{pr.203}
Suppose that there exist additive mappings $\varphi, \psi:\mathscr{F}\longrightarrow \mathscr{E}$
such that $\big\langle \varphi(z), \psi(w)\big\rangle=0$ and
$a \big\langle \varphi(z), \varphi(w)\big\rangle a^\ast = (1 - a) \big\langle \psi(z), \psi(w)\big\rangle (1 - a)^\ast$
for all $z, w\in \mathscr{F}$. If $f:\mathscr{E}\longrightarrow \mathscr{G}$
is an odd orthogonally $a$-Jensen mapping,
then $f$ is additive on $\mathscr{K} := \varphi(\mathscr{F})+\psi(\mathscr{F})$.
\end{proposition}
\begin{proof}
Since $f$ is odd $f(0)=0$. Thus for every $x, y\in \mathscr{F}$, by Lemma \ref{lm.202} we conclude that
\begin{multline}\label{id.204}
a f\big(\varphi(x) + \varphi(y)\big) + (1 - a) f\big(\psi(x)-\psi(y)\big) =
\\a f\big(\varphi(x)\big) + (1 - a) f\big(\psi(x)\big) + a f\big(\varphi(y)\big) + (1 - a) f\big(-\psi(y)\big).
\end{multline}
Switching $x$ and $y$ in (\ref{id.204}) we obtain
\begin{multline}\label{id.205}
a f\big(\varphi(y) + \varphi(x)\big) + (1 - a) f\big(\psi(y)-\psi(x)\big) =
\\a f\big(\varphi(y)\big) + (1 - a) f\big(\psi(y)\big) + a f\big(\varphi(x)\big) + (1 - a) f\big(-\psi(x)\big).
\end{multline}
Add (\ref{id.204}) and (\ref{id.205}) and use the fact that $f$ is odd to get
\begin{align*}
2a f\big(\varphi(x) + \varphi(y)\big) = 2a f\big(\varphi(x)\big) + 2a f\big(\varphi(y)\big),
\end{align*}
or equivalently,
\begin{align*}
f\big(\varphi(x)+\varphi(y)\big)=f\big(\varphi(x)\big)+f\big(\varphi(y)\big).
\end{align*}
Hence $f$ is additive on $\varphi(\mathscr{F})$. Similarly $f$ is additive on $\psi(\mathscr{F})$.
Now for every $z_1, z_2\in \mathscr{K}$ there exist $x_1, x_2, y_1, y_2\in \mathscr{F}$ such that
$$z_1=\varphi(x_1)+\psi(y_1) \quad \mbox{and} \quad z_2=\varphi(x_2)+\psi(y_2).$$
We have
\begin{align*}
f(z_1+z_2)&= f\big(\varphi(x_1+x_2)+\psi(y_1+y_2)\big)
\\&=f\Big(aa^{-1} \varphi(x_1+x_2) + (1 - a)(1 - a)^{-1}\psi(y_1+y_2)\Big)
\\&= a f\big(a^{-1} \varphi(x_1+x_2)\big) + (1 - a) f\big((1 - a)^{-1} \psi(y_1+y_2) \big)
\\&\hspace{2cm}\Big(\mbox{since}\,\,\big\langle a^{-1} \varphi(x_1+x_2), (1 - a)^{-1} \psi(y_1+y_2)\big\rangle=0\,\,
\\&\hspace{7cm}\mbox{and $f$ is orthogonally $a$-Jensen}\Big)
\\&=f\big(\varphi(x_1+x_2)\big)+f\big(\psi(y_1+y_2)\big)\hspace{2cm}\Big(\mbox{by Lemma \ref{lm.201}\,(i) and (ii)}\Big)
\\&= f\big(\varphi(x_1)\big)+f\big(\varphi(x_2)\big)+f\big(\psi(y_1)\big)+f\big(\psi(y_2)\big)
\\&\hspace{5cm}\Big(\mbox{by the additivity of $f$ on $\varphi(\mathscr{F})$ and $\psi(\mathscr{F})$}\Big)
\\&= a f\big(a^{-1} \varphi(x_1)\big)+ (1 - a) f\big((1 - a)^{-1} \psi(y_1)\big)
\\&\hspace{3cm}+ a f\big(a^{-1} \varphi(x_2)\big) + (1 - a) f\big((1 - a)^{-1}\psi(y_2)\big)
\\&\hspace{8cm}\Big(\mbox{by Lemma \ref{lm.201}\,(i) and (ii)}\Big)
\\&= f\big(\varphi(x_1)+\psi(y_1)\big)+f\big(\varphi(x_2)+\psi(y_2)\big)
\\&\hspace{3cm}\Big(\mbox{since}\,\,\big\langle a^{-1} \varphi(x_1), (1 - a)^{-1}\psi(y_1)\big\rangle = 0,
\\&\hspace{5cm}\big\langle a^{-1}\varphi(x_2), (1 - a)^{-1}\psi(y_2)\big\rangle = 0
\\&\hspace{7cm}\mbox{and $f$ is orthogonally $a$-Jensen}\Big)
\\&= f(z_1)+f(z_2).
\end{align*}
Thus $f$ is additive on $\mathscr{K}$.
\end{proof}
\begin{proposition}\label{pr.210}
Suppose that there exist additive mappings $\varphi, \psi:\mathscr{F}\longrightarrow \mathscr{E}$
such that $\big\langle \varphi(z), \psi(w)\big\rangle=0$ and
$a \big\langle \varphi(z), \varphi(w)\big\rangle a^\ast = (1 - a) \big\langle \psi(z), \psi(w)\big\rangle (1 - a)^\ast$
for all $z, w\in \mathscr{F}$. If $f:\mathscr{E}\longrightarrow \mathscr{G}$ is an even orthogonally
$a$-Jensen mapping such that $f(0)=0$, then f is quadratic on $\mathscr{K} := \varphi(\mathscr{F})+\psi(\mathscr{F})$.
\end{proposition}
\begin{proof}
Since $f$ is even and $f(0)=0$, putting $x=y$ in Lemma \ref{lm.202} we infer that
$$a f\big(2\varphi(x)\big) = 2a f\big(\varphi(x)\big) + 2(1 - a) f\big(\psi(x)\big) \,\qquad (x\in \mathscr{F}).$$
Similarly, we have
$$(1 - a) f\big(2\psi(x)\big) = 2a f\big(\varphi(x)\big) + 2(1 - a) f\big(\psi(x)\big) \,\qquad (x\in \mathscr{F}).$$
Therefore, we conclude that
\begin{align}\label{id.211}
a f\big(2\varphi(x)\big) = (1 - a) f\big(2\psi(x)\big) \,\qquad (x\in \mathscr{F}).
\end{align}
If we put $\frac{x}{2}$ instead of $x$ in (\ref{id.211}) we get
\begin{align}\label{id.212}
a f\big(\varphi(x)\big) = (1 - a) f\big(\psi(x)\big) \,\qquad (x\in \mathscr{F}).
\end{align}
Now for every $x, y\in \mathscr{F}$ we have
\begin{align*}
f\big(\varphi(x)&+\varphi(y)\big) + f\big(\varphi(x)-\varphi(y)\big)
\\&= f\big(\varphi(x)+\varphi(y)\big) + a^{-1}(1 - a) f\big(\psi(x)-\psi(y)\big) \hspace{4.2cm}\Big(\mbox{by}\,(\ref{id.212})\Big)
\\&= f\big(\varphi(x)\big) + a^{-1}(1 - a) f\big(\psi(x)\big) + f\big(\varphi(y)\big) + a^{-1}(1 - a) f\big(\psi(y)\big)
\\& \hspace{11cm}\Big(\mbox{by Lemma}\,\,\ref{lm.202}\Big)
\\&= f\big(\varphi(x)\big) + f\big(\varphi(x)\big) + f\big(\varphi(y)\big) + f\big(\varphi(y)\big)
\hspace{4.9cm}\Big(\mbox{by}\,\,(\ref{id.212})\Big)
\\&= 2f\big(\varphi(x)\big) + 2f\big(\varphi(y)\big).
\end{align*}
Thus
\begin{align*}
f\big(\varphi(x)+\varphi(y)\big)+f\big(\varphi(x)-\varphi(y)\big)=2f\big(\varphi(x)\big)+2f\big(\varphi(y)\big).
\end{align*}
So $f$ is quadratic on $\varphi(\mathscr{F})$.
Similarly $f$ is quadratic on $\psi(\mathscr{F})$.

By the same reasoning as in the last part of Proposition \ref{pr.203}
we conclude that $f$ is quadratic on $\mathscr{K}$.
\end{proof}
We are now in position to establish the main result.
If $A:\mathscr{E}\longrightarrow \mathscr{G}$ is $a$-additive and
$B:\mathscr{E}\times \mathscr{E}\longrightarrow \mathscr{G}$ is
$a$-biadditive orthogonality preserving, then the mapping
$f:\mathscr{E}\longrightarrow \mathscr{G}$ defined by
$$f(x) = A(x) + B(x, x) + f(0) \qquad (x\in \mathscr{E})$$
is an orthogonally $a$-Jensen mapping. Namely, if $\langle x, y\rangle=0$ then
\begin{align*}
f\big(ax + (1 - a)y\big)&= A\big(ax + (1 - a)y\big) + B\big(ax + (1 - a)y, ax + (1 - a)y\big) + f(0)
\\&= a A(x ) + a B(x, x) + a f(0)
\\&\hspace{4cm}+ (1 - a) A(y) + (1 - a)B(y, y) + (1 - a) f(0)
\\&= a f(x) + (1 - a) f(x).
\end{align*}
The following theorem is a kind of converse of the previous discussion.
\begin{theorem}\label{th.217}
Let $a$ be an element of a unital $C^{*}$-algebra $\mathfrak{A}$
such that $a, 1-a$ are invertible and let $\mathscr{E}, \mathscr{F}, \mathscr{G}$
be inner product $\mathfrak{A}$-modules. Suppose that there exist additive mappings
$\varphi, \psi:\mathscr{F}\longrightarrow \mathscr{E}$ such that
$\big\langle \varphi(z), \psi(w)\big\rangle=0$ and
$a \big\langle \varphi(z), \varphi(w)\big\rangle a^\ast = (1 - a) \big\langle \psi(z), \psi(w)\big\rangle (1 - a)^\ast$
for all $z, w\in \mathscr{F}$. Let $\mathscr{K} := \varphi(\mathscr{F})+\psi(\mathscr{F})$.
If $f:\mathscr{E}\longrightarrow \mathscr{G}$ is an orthogonally $a$-Jensen mapping,
then there exist unique mappings $A:\mathscr{E}\longrightarrow \mathscr{G}$ and
$B:\mathscr{E}\times \mathscr{E}\longrightarrow \mathscr{G}$ such that
$A$ is $a$-additive on $\mathscr{K}$, $B$ is symmetric $a$-biadditive orthogonality preserving on
$\mathscr{K}\times \mathscr{K}$ and
$$f(x) = A(x) + B(x, x) + f(0) \qquad (x\in \mathscr{K}).$$
Moreover, $B(x,y)=\frac{1}{8} \Big(f(x+y)+f(-x - y) - f(x-y) -f(-x + y)\Big)$
and $A(x)=\frac{1}{2} \Big(f(x)-f(-x)\Big)$ for every $x, y \in \mathscr{E}$.
\end{theorem}
\begin{proof}
By passing to $f-f(0)$, if necessary, we may assume that $f(0)=0$.
We decompose $f$ into its even and odd parts by
$$f_o(x) =\frac{1}{2}\big(f(x)-f(-x)\big) \quad \mbox{and} \quad f_e(x) =\frac{1}{2}\big(f(x)+f(-x)\big)$$
for all $x\in \mathscr{E}$.
Set $A(x) := f_o(x)$ and $B(x, y) :=\frac{1}{4}\big(f_e(x+y)-f_e(x-y)\big)$.
It is easy to show that $f_o$ is odd orthogonally $a$-Jensen.
So by Proposition \ref{pr.203}, $f_o$ is additive on $\mathscr{K}$.
Also, by Lemma \ref{lm.201} (v), we have
$$A(ax) = f_o(ax) = a f_o(x) + b f_o(0) = a A(x) \qquad (x\in \mathscr{K}).$$
Hence $A$ is $a$-additive on $\mathscr{K}$.
Furthermore, $f_e$ is even orthogonally $a$-Jensen.
Since $f_e(0)= f(0)-f_o(0)=0$, $f_e$ is quadratic on $\mathscr{K}$ by Proposition \ref{pr.210}.
Thus $f_e(2x)=4f_e(x)$ and so,
$$B(x, x)=\frac{1}{4}\big(f_e(2x)-f_e(0)\big)=f_e(x).$$
This implies
$$f(x)=f_o(x)+f_e(x)=A(x)+B(x, x) \qquad (x\in \mathscr{K}).$$
For each $x, y, z\in \mathscr{E}$ we have
\begin{align}\label{id.218}
B(x+y, 2z)&= \frac{1}{4}\Big(f_e(x+y+2z)-f_e(x+y-2z)\Big)\nonumber
\\&=\frac{1}{4}\Big(f_e\big((x+z)+(y+z)\big)+f_e\big((x+z)-(y+z)\big)\nonumber
\\&\hspace{2cm}-f_e\big((x-z)+(y-z)\big)-f_e\big((x-z)-(y-z)\big)\Big)\nonumber
\\&=\frac{1}{4}\Big(2f_e(x+z)+2f_e(y+z)-2f_e(x-z)-2f_e(y-z)\Big)\nonumber
\\&=\frac{1}{2}\Big(f_e(x+z) - f_e(x-z) + f_e(y+z) - f_e(y-z)\Big)\nonumber
\\&=2B(x, z)+2B(y, z).
\end{align}
In particular, by choosing $y=0$, we get
\begin{align}\label{id.219}
B(x, 2z)&=2B(x, z)+2B(0, z)\nonumber
\\&=2B(x, z)+\frac{1}{4}\Big(f(x)+f(-x)-f(x)-f(-x)\Big)\nonumber
\\&=2B(x, z).
\end{align}
If we replace $x$ by $x+y$ in (\ref{id.219}), then by (\ref{id.218}) we obtain
$$B(x+y, z)=\frac{1}{2}B(x+y, 2z)=\frac{1}{2}[2B(x, z)+2B(y, z)]=B(x, z)+B(y, z),$$
and similarly $B(x, y+z)=B(x, y)+B(x, z)$. Therefore $B$ is biadditive on $\mathscr{K}\times \mathscr{K}$.
Also, by Lemma \ref{lm.201} (vi), we have
$$B(ax, ax) = f_e(ax) = a f_e(x) + (1 - a) f_e(0) = a B(x, x) \qquad (x\in \mathscr{K})$$
and analogously
$$B((1 - a)x, (1 - a)x) = (1 - a) B(x, x) \qquad (x\in \mathscr{K}).$$
Hence, $B$ is $a$-biadditive on $\mathscr{K}\times \mathscr{K}$.
Further, for each $x, y\in \mathscr{K}$, it follows from $\langle x, y\rangle=0$ that
\begin{align*}
B(x, y)&=\frac{1}{2}\big(B(x, y)+B(y, x)\big)
\\&=\frac{1}{2}\big(B(x+y, x+y)-B(x, x)-B(y, y)\big)
\\&=\frac{1}{2}\big(f(x+y)-A(x+y)\big)-\frac{1}{2}\big(f(x)-A(x)\big)-\frac{1}{2}\big(f(y)-A(y)\big)
\\&=\frac{1}{2}\big(f(x+y)-f(x)-f(y)\big) \hspace{4.8cm}\Big(\mbox{$A$ is additive on $\mathscr{K}$}\Big)
\\&=\frac{1}{2}\Big(f\big(aa^{-1} x + (1 - a)(1 - a)^{-1} y \big)-f(x)-f(y)\Big)
\\&=\frac{1}{2}\Big(a f(a^{-1} x) + (1 - a) f((1 - a)^{-1}y)\Big) - \frac{1}{2}f(x)-\frac{1}{2}f(y)
\\&\hspace{3.1cm}\Big(\mbox{since}\,\,\big\langle a^{-1} x, (1 - a)^{-1}y\big\rangle = 0 \,\,\mbox{and $f$ is orthogonally $a$-Jensen}\Big)
\\&=\frac{1}{2}f(x)+\frac{1}{2}f(y)-\frac{1}{2}f(x)-\frac{1}{2}f(y)\hspace{3cm}\Big(\mbox{by Lemma \ref{lm.201}\,(i) and (ii)}\Big)
\\&=0.
\end{align*}
Also, since $f_e$ is even, $B$ is symmetric.
Thus $B$ is $a$-biadditive orthogonality preserving on $\mathscr{K}\times \mathscr{K}$.
Finally suppose, $f(x) = A_1(x)+B_1(x,x)+f(0) = A_2(x)+B_2(x,x)+f(0)$ for any $x$ for
the specified kind of mappings $A$ and $B$. Hence, $A_1(x)-A_2(x)=B_1(x,x)-B_2(x,x)$ for any $x$.
However, the left part is an odd mapping, and the right part is an even mapping.
So both these terms are equal to zero for any $x$.
Thus we conclude that $A$ and $B$ are uniquely determined by $f$.
\end{proof}
\begin{remark}
The $a$-additive mappings $\varphi$, $\psi$ from $\mathscr {F}$ to $\mathscr {E}$ need not to be injective.
Also, the linear span of their ranges need not coincide with $\mathscr {E}$.
So, $a$ and $1-a$ might be assumed merely to admit generalized inverses inside the $C^*$-algebra $\mathfrak A$.
By the requested equality
$a \big\langle \varphi(z), \varphi(w) \big\rangle a^* = (1-a) \big\langle \psi(z), \psi(w) \big\rangle (1-a)^*$
for any $z, w \in \mathscr {F}$ the range projections of $a$ and of $1-a$ in the bidual von Neumann algebra
${\mathfrak A}^{**}$ of $\mathfrak A$ have to coincide.
The domain projections of $a$ and of $1-a$ in the bidual von Neumann algebra
${\mathfrak A}^{**}$ of $\mathfrak A$ have to majorize the support projection of the subset
$\big\langle \varphi(\mathscr {F}), \varphi(\mathscr {F})\big\rangle$ and of the subset
$\big\langle \psi(\mathscr {F}), \psi(\mathscr {F})\big\rangle$ in ${\mathfrak A}^{**}$, respectively.
For $a$ and $1-a$ admitting generalized inverses in $\mathfrak A$ the domain
and the range projections belong to ${\mathfrak A} \subseteq {\mathfrak A}^{**}$.
However, they might not belong to the center of ${\mathfrak A}$,
so the Hilbert $C^*$-modules $\mathscr {F}$ and $\mathscr {E}$
cannot be reduced appropriately compatible with their module structure, in general.
\end{remark}
In the following result we obtain the representation of orthogonally
$p$-Jensen mappings in inner product modules.
\begin{corollary}\label{cr.216}
Let $p\in(0, 1)$ be rational. Suppose that there exist
additive mappings $\varphi, \psi:\mathscr{F}\longrightarrow \mathscr{E}$
such that $\big\langle \varphi(z), \psi(w)\big\rangle=0$
and $(1-p)^2\big\langle \varphi(z), \varphi(w)\rangle = p^2\big\langle \psi(z), \psi(w)\big\rangle$
for all $z, w\in \mathscr{F}$. If $f:\mathscr{E}\longrightarrow \mathscr{G}$ is orthogonally $p$-Jensen,
then there exists a unique mapping $A:\mathscr{E}\longrightarrow \mathscr{G}$
such that $A$ is additive on $\mathscr{K} := \varphi(\mathscr{F})+\psi(\mathscr{F})$ and
$$f(x) = A(x) + f(0) \qquad (x\in \mathscr{K}).$$
\end{corollary}
\begin{proof}
By Theorem \ref{th.217}, there exist unique mappings $A:\mathscr{E}\longrightarrow \mathscr{G}$
and $B:\mathscr{E}\times \mathscr{E}\longrightarrow \mathscr{G}$ such that
$A$ is additive on $\mathscr{K}$, $B$ is symmetric biadditive orthogonality
preserving on $\mathscr{K}\times \mathscr{K}$ and
$$f(x) = A(x) + B(x, x) +f(0) \qquad (x\in \mathscr{K}).$$
We have
\begin{align*}
A(x) + B(x, x) &+f(0) = f(x)
\\&= (1-p) f(0) + p f(\frac{1}{p} x)\hspace{5cm}\Big(\mbox{by Lemma}\,\ref{lm.201}\Big)
\\&= (1-p) f(0) + p \Big( A(\frac{1}{p} x) + B(\frac{1}{p} x,\frac{1}{p} x)+f(0)\Big)
\hspace{1cm}\Big(\mbox{by Theorem}\,\ref{th.217}\Big)
\\&=A(x)+\frac{1}{p}B(x,x)+f(0),\hspace{0.5cm}\Big(\mbox{by $\mathbb Q$-linearity A and $\mathbb{Q}$-bilinearity B}\Big)
\end{align*}
for any $x \in \mathscr{K}$. Consequently, $(1-\frac{1}{p}) B(x,x)=0$
and therefore, $B(x,x)=0$ for all $x \in \mathscr{K}$.
Thus $B=0$ on $\mathscr{K}$ and $f(x) = A(x) + f(0)$ on $\mathscr{K}$.

Note that if $A(x)=0$ for some $x \in \mathscr{K}$ then by $\mathbb{Q}$-linearity of
$A$ we reach $A(qx) = 0$ for all rational numbers $q$.
So, $f(qx) = A(qx) + f(0) = f(0)$ for all rational numbers $q$.
\end{proof}
\begin{corollary}\label{cr.217}
Let $\mathscr{F}$ be a submodule of $\mathscr{E}$,
and $\varphi:\mathscr{F}\longrightarrow \mathscr{E}$ be a morphism such that
$\varphi(\mathscr{F} )\subseteq \mathscr{F}^\perp$.
If $f:\mathscr{E}\longrightarrow \mathscr{G}$ is orthogonally Jensen,
then there exists a unique mapping $A:\mathscr{E}\longrightarrow \mathscr{G}$
such that $A$ is additive on $\mathscr{K} := \mathscr{F}\oplus\varphi(\mathscr{F})$ and
$$f(x) = A(x) + f(0) \qquad (x\in \mathscr{K}).$$
\end{corollary}
\begin{proof}
Let $id:\mathscr{F}\longrightarrow \mathscr{F}$ be the
identity mapping. Since $\varphi$ is a morphism and $\varphi(\mathscr{F} )\subseteq \mathscr{F}^\perp$,
for every $x, y\in \mathscr{F}$ we obtain
$\big\langle \varphi(x), \varphi(y)\big\rangle = \big\langle id(x), id(y)\big\rangle$
and $\big\langle \varphi(x), id(y)\big\rangle=0$. It remains to apply Corollary \ref{cr.216} for $p=\frac{1}{2}$.
\end{proof}
\begin{corollary}\label{co.218}
Let $\mathscr{F}$ be a fully complemented submodule of $\mathscr{E}$.
If $f:\mathscr{E}\longrightarrow \mathscr{G}$ be orthogonally Jensen,
then there exists a unique mapping $A:\mathscr{E}\longrightarrow \mathscr{G}$
such that $A$ is additive on $\mathscr{F}$ and
$f(x) = A(x) + f(0)$ for every $x\in \mathscr{F}$.
\end{corollary}
\begin{proof}
Since $\mathscr{F}$ is a fully complemented submodule of $\mathscr{E}$,
therefore $\mathscr{F}\oplus \mathscr{F}^\perp=\mathscr{E}$
and $\mathscr{F}^\perp\thicksim \mathscr{E}$. So, there exists an adjointable mapping
$\phi:\mathscr{E}\longrightarrow \mathscr{F}^\perp$ such that $\phi^* \phi=id_\mathscr{E}$.
We have $\big\langle \phi(x), \phi(y)\big\rangle = \big\langle \phi^* \phi(x), y\big\rangle = \big\langle x, y\big\rangle$
for all $x, y\in \mathscr{E}$. Let us define $\varphi := \phi|_\mathscr{F}$ .
Then $\varphi(\mathscr{F} )\subseteq \mathscr{F}^\perp$ and
$\big\langle \varphi(x), \varphi(y)\big\rangle = \big\langle x, y\big\rangle$
for all $x, y\in \mathscr{F}$. Thus $\varphi:\mathscr{F}\longrightarrow \mathscr{E}$
is a morphism such that $\varphi(\mathscr{F} )\subseteq \mathscr{F}^\perp$.
Now the statement follows from Corollary \ref{cr.217}.
\end{proof}
\textbf{Acknowledgement.}
The author would like to thank the referee for her/his valuable suggestions and comments.
He would also like to thank Professor M. S. Moslehian and Professor M. Frank for their 
invaluable suggestions while writing this paper.
\bibliographystyle{amsplain}

\end{document}